\documentclass[11pt]{amsart}
\usepackage{}
\usepackage{amsfonts}
\usepackage{amsfonts,amssymb}
\usepackage{srcltx} 

\usepackage{mathrsfs}
\let\mathcal\mathscr

\usepackage[all,ps,cmtip]{xy}

\def\llra{\hbox to 10mm{\rightarrowfill}}

\def\lllra{\hbox to 15mm{\rightarrowfill}}

\def\PA{{\widehat A}}
\def\PB{{\widehat B}}
\def\PK{{\widehat K}}

\def\PE{{\widehat E}}

\def\PT{{\widehat T}}

\def\Pp{{\widehat\pi}}

\def\phi{{\varphi}}

\def\wA{{\widetilde A}}
\def\wa{{\widetilde a}}

\def\wX{{\widetilde X}}

\def\G2{\mathbf{Z}/2\mathbf{Z}\times\mathbf{Z}/2\mathbf{Z}}

\def\cI{\mathcal{I}}

\def\cF{\mathcal{F}}

\def\cO{\mathcal{O}}

\def\cJ{\mathcal{J}}

\def\cU{\mathcal{U}}
\def\cV{\mathcal{V}}

\DeclareMathOperator{\rank}{rank}

\DeclareMathOperator{\codim}{codim}
\DeclareMathOperator{\Pic}{Pic}

\DeclareMathOperator{\pr}{pr}

\DeclareMathOperator{\Bs}{Bs}

\newtheorem{lemm}{Lemma}[section]
\newtheorem{theo}[lemm]{Theorem}

\newtheorem{prop}[lemm]{Proposition}

\newtheorem*{conj*}{Conjecture}

\theoremstyle{definition}
\newtheorem{defi}[lemm]{Definition}
\newtheorem{rema}[lemm]{Remark}

\newtheorem{exam}[lemm]{Example}

\newenvironment{proof-main}{\trivlist \item[\hskip\labelsep{\sc
\textbf{Proof of Theorem \ref{theo}.}}]\rm}{\hbox
to.1pt{\hss}\hfill$\square$\bigskip\endtrivlist}
\theoremstyle{remark}
\newtheorem*{remark*}{Remark}
\newtheorem*{note*}{Note}

\def\moins{\mathop{\hbox{\vrule height 3pt depth -2pt
width 5pt}\,}}

\begin{document}
\title[Varieties of Albanese fibre dimension one]{Cohomological support loci of varieties of Albanese fiber dimension one}

\author{Zhi Jiang}
\address{Math\'ematiques B\^atiment 425, Universit\'e Paris-Sud, F-91405 Orsay, France}

\email{Zhi.Jiang@math.u-psud.fr}
\thanks{}

\author{Hao Sun}
\address{Department of Mathematics, Huazhong Normal University, Wuhan 430079, People's Republic of China}

\email{hsun@mail.ccnu.edu.cn}
\thanks{The second author was partially supported by the Mathematical Tianyuan Foundation of China
(No. 11126192).}



\date{\today}


\begin{abstract}
Let $X$ be a smooth projective variety of Albanese fiber dimension 1
and of general type. We prove that the translates through $0$ of all
components of  $V^0(\omega_X)$ generate $\Pic^0(X)$. We then study
the pluricanonical maps of $X$. We show that $|4K_X|$ induces a
birational map of $X$.
\end{abstract}

\maketitle

\section{Introduction}
In \cite[Theorem 1]{CH}, Chen and Hacon proved that if $X$ is of maximal
Albanese dimension and of general type, then the translates through
$0$ of all irreducible components of  $V^0(\omega_X)$ generate $\Pic^0(X)$. This
is a fundamental result to study pluricanonical maps of
varieties of maximal Albanese dimension (cf. \cite{J, T, JLT}). Here
we provide a similar result for varieties of general type and of
Albanese fiber dimension one:

\begin{theo}\label{theo} Let $X$ be a smooth projective variety of dimension $\geq 2$,
of Albanese fiber dimension one and of general type.
Then the translates through $0$ of all irreducible components of $V^0(\omega_X)$ generate $\Pic^0(X)$.
\end{theo}

We notice that this kind of result applies only for varieties of
Albanese fiber dimension $\leq 1$. Indeed, we take $X=Y\times Z$,
where $Y$ is a variety of maximal Albanese dimension and of general
type and $Z$ is a variety of general type of dimension $l\geq 2$
with $p_g(Z)=q(Z)=0$. Then $X$ is of Albanese fiber dimension $l$
but $V^0(\omega_X)$ is empty. As an application of Theorem
\ref{theo}, we give an improvement of a result of Chen and Hacon
\cite[Corollary 4.3]{ch2}:
\begin{theo}\label{theorem1.1}
Let $X$ be a smooth projective variety of Albanese fiber dimension
one and of general type. Then the $4$-canonical map $\varphi_{|4K_X|}$
is birational.
\end{theo}
In Section 2, we recall some definitions and results which will be used in this paper. We will prove
Theorem \ref{theo} in Section 3 and extend it to non-general type case in Section 4. We then study  pluricanonical
maps of varieties of Albanese fiber dimension 1 and prove Theorem
\ref{theorem1.1} in Section 5.

\subsection*{Acknowledgements} The first author thanks Prof. Jungkai A. Chen for various comments and pointing out an error in an earlier version of this paper.
The second author would like to thank Dr.
Lei Zhang for many useful discussions. Finally, the authors are grateful to the referee for carefully reading
the paper and for many suggestions that improve the exposition.

\section{Notations and Preliminaries}
 In this section, we recall some notations and techniques that will be
used throughout the paper.

We will always denote by $X$ a smooth complex projective variety
and $a_X: X\rightarrow A_X$  its Albanese morphism. If $l=\dim X-\dim a_X(X)$, then we say that $X$ is of Albanese fiber
dimension $l$.  In particuler, $X$ is of Albanese fiber dimension $0$ if and
only if $X$ is of maximal Albanese dimension.

For an abelian variety $A$, we
will frequently denote by $\PA=\Pic^0(A)$ the dual abelian variety. Moreover,
for a morphism $t: A\rightarrow B$ between abelian varieties, we
will denote by $\widehat{t}: \PB\rightarrow \PA$ the dual morphism
between the dual abelian varieties. By abuse of notation, we  identify a point of $\PA$ with the corresponding line bundle of $A$.

Let $t: X\rightarrow A$ be a morphism from $X$ to an abelian variety and
let $\cF$ be a coherent sheaf on $X$, we will denote by $V^i(\cF, t)=\{P\in\PA\mid h^i(X, \cF\otimes t^*P)>0\}$ the $i$-th cohomological support locus of $\cF$.
In particular, if $t=a_X$, we will simply denote $V^i(\cF, a_X)$ by $V^i(\cF)$.

When  $\cF=\omega_X$, we have the following important theorem of Green-Lazarsfeld and Simpson (see \cite[Theorem 0.1]{gl2} and \cite[Theorem 4.2]{sim}):
\begin{theo}\label{structure}
For any $i\geq 0$, each irreducible component of $V^i(\omega_X, t)$ is a torsion translate of a subtorus of $\PA$.
\end{theo}

We recall the following definitions due to Pareschi and Popa (\cite{pp1} and \cite{pp2}).

\begin{defi}
Let $\mathcal{F}$ be a coherent sheaf on a smooth projective variety
$Y$.
\begin{enumerate}
\item The sheaf $\mathcal{F}$ is said to be a $GV$-sheaf if $\codim
V^i(\mathcal{F})\geq i$ for all $i\geq0$.
\item If $Y$ is an abelian variety, $\mathcal{F}$ is said to be $M$-regular if $\codim
V^i(\mathcal{F})>i$ for every $i>0$.
\end{enumerate}
\end{defi}

\begin{lemm}\label{lemmaPP}
Assume that $\mathcal{F}$ is a $GV$-sheaf on a smooth projective variety
$Y$. Let $W$ be an irreducible component of $V^0(\mathcal{F})$ such that
$k=\codim_{\Pic^0(Y)}W$. Then $W$ is also a component of
$V^k(\mathcal{F})$. In particular, $\dim X\geq k$.
\end{lemm}
\begin{proof}
Please see \cite[Proposition 3.15]{pp2}.
\end{proof}


The following definition is also introduced by Pareschi and Popa (\cite[Definition 2.10]{pp1}).
\begin{defi}\label{definition}
Let $\mathcal{F}$ be a coherent sheaf on a smooth projective variety
$Y$. We say $\mathcal{F}$ is continuously globally generated at $y\in Y$ (in brief CGG at $y$)
if the nature map $$\bigoplus_{P\in U}
H^0(Y, \mathcal{F}\otimes P)\otimes{P}^{\vee}\rightarrow
\mathcal{F}\otimes \mathbb{C}(y)$$ is surjective for any non-empty
open subset $U\subset{\Pic}^0(Y)$.
\end{defi}

We will need the following theorem (\cite[Proposition 2.13]{pp1}):
\begin{theo}\label{pareschi-popa}
Let $\cF$ be a M-regular sheaf on an abelian variety $A$. Then $\cF$ is CGG at each point of $A$.
\end{theo}

The following proposition is a slight generalization of \cite[Theorem 0.1 (2)]{gl2}.
\begin{prop}\label{fiberdimension}
 Let $\alpha: X\rightarrow A$ be a morphism from a smooth porjective variety to an abelian variety such that $\dim X-\dim \alpha(X)=s$.
Assume that $V^0(\omega_X, \alpha)$ has an irreducible component $P_0+\PB$ of
 codimension $0\leq k<\infty$ in $\PA$, where $P_0$ is a torsion point and $\PB$ is an abelian subvariety of $\PA$.
Then there exists a commutative diagram
\begin{eqnarray*}
 \xymatrix{X\ar[r]^{\alpha}\ar[d]_f & A \ar[d]^{\pi}\\
X_B\ar[r]^g& B,}
\end{eqnarray*}
where $\pi$ is the natural projection, $f$ is a fibration and $X_B$ is a normal variety of dimension $\dim X-s-k$.
\end{prop}

\begin{proof}
By assumption, $P_0+\PB$ is an irreducible component of $V^0(\omega_X, \alpha)=V^0(\alpha_*\omega_X)$.
We know that $\alpha_*\omega_X$ is a $GV$-sheaf on $A$ (see \cite[Corollary 4.2]{Hac}). Hence by Lemma \ref{lemmaPP}, $P_0+\PB$ is an irreducible component
of $V^k(\alpha_{*}\omega_X)$.

 Let $X\xrightarrow{h} X_A\xrightarrow{t} A $ (resp. $X_A\xrightarrow{h_B} X_B\xrightarrow{g} B$) be the Stein factorization of $\alpha$
(resp. the Stein factorization of the composition $\pi\circ t$). We then have the following commutative diagram
\begin{eqnarray*}
\xymatrix{
X\ar[dr]^{\alpha}\ar[d]^h\ar@/_2pc/[dd]_{f}\\
X_A\ar[r]^t\ar[d]^{h_B} & A\ar[d]^{\pi}\\
X_B\ar[r]^g& B.}
\end{eqnarray*}

Since $t$ is finite, $H^k(A, \alpha_{*}\omega_X\otimes Q)=H^k(X_A, h_*\omega_X\otimes t^*Q)\neq 0$, for any $Q\in P_0+\PB$. On the other hand,
by Koll\'ar's theorem
(\cite[Theorem 3.4]{kollar2})
we have
\begin{eqnarray*}
h^k(X_A, h_*\omega_X\otimes t^*(P_0\otimes P))&= &\sum_{i+j=k }h^i(X_B, R^jh_{B*}(h_*(\omega_X\otimes P_0))\otimes g^*P)\\&\neq&0,
\end{eqnarray*}
for all $P\in \PB$. For all $j\geq 0$, since $P_0$ is a torsion line bundle, the sheaves $R^jh_{B*}(h_*(\omega_X\otimes P_0))$ are $GV$-sheaves
(see \cite[Theorem 2.2]{HP}), hence
$R^kh_{B*}(h_*(\omega_X\otimes P_0))\neq 0$. We conclude again by Koll\'ar's theorem
(\cite[Theorem 3.4]{kollar2}) that $\dim X_A-\dim X_B\geq k$ and hence the equality holds and $\dim X_B=\dim X-s-k$.
\end{proof}

 \section{Proof of Theorem \ref{theo}}
\begin{lemm}\label{lemma 1}Let $Y$ be a smooth projective variety of general type and of maximal Albanese dimension.
Let $t: Y\rightarrow A$ be a  morphism to an abelian variety $A$ of dimension $\geq 1$ such that $\dim Y-\dim t(Y)=1$.
Then $$\dim V^0(\omega_Y, t)\geq 1,$$ where $$V^0(\omega_Y, t):=\{P\in\PA\mid H^0(Y, \omega_Y\otimes t^*P)\neq 0\},$$ and $\dim V^0(\omega_Y, t)$
is by definition the maximal dimension of an irreducible component of $V^0(\omega_Y, t)$.
\end{lemm}
\begin{proof}By the universal property of the Albanese moprhism $a_Y$,
we have a commutative diagram
\begin{eqnarray*}
\xymatrix{
Y\ar[r]^{a_Y}\ar[dr]_t & A_Y\ar[d]^{\mu}\\
& A.}
\end{eqnarray*}

Since $K_Y$ is an effective divisor (see \cite[Lemma 3.1]{CH1} or \cite[Lemma 2.3]{JLT}), we may assume that $\mu$ is surjective,
otherwise $\ker(\widehat{\mu}:\PA\rightarrow\PA_Y)$ is contained in
$V^0(\omega_X, t)$ and is of dimension $\geq 1$.

When $\mu$ is surjective, $\widehat{\mu}:
\PA\rightarrow \PA_X$ is an isogeny onto its image and $$\dim
V^0(\omega_Y, t)=\dim(\widehat{\mu}(\PA)\cap V^0(\omega_Y)).$$ Hence
if $V^0(\omega_Y)=\PA_Y$, Lemma \ref{lemma 1} is clear. Otherwise,
$\chi(Y, \omega_Y)=0$ and $V^0(\omega_Y)$ is a union of torsion
translates of proper abelian subvarieties of $\PA_Y$ and $\dim
V^0(\omega_Y)\geq 1$ (see \cite[Theorem 1]{CH}).

Let $K$ be the neutral component of the kernel of
$A_{Y}\xrightarrow{\mu} A$. We then have the exact  sequence
$$\widehat{A}\xrightarrow{\widehat{\mu}} \PA_{Y}\twoheadrightarrow
\PK.$$ Because an irreducible component of a general fiber of $t$ is
a smooth curve of genus $\geq 2$, hence for any $P\in \PA_{Y}$,
$t_{*}(\omega_X\otimes P)$ is a non-zero $GV$-sheaf on $A$ by \cite[Corollary 4.2]{Hac}\footnote{Hacon shows that $t_*\omega_X$ is a $GV$-sheaf
and the twisted version follows the same way by combining \cite[Theorem 1.2]{Hac} and \cite[Corollary 10.15]{K1}.}.
Therefore for any $P\in\PA_{Y}$, there exists a line bundle
$P'\in\widehat{A}$ such that $$H^0(Y, \omega_{Y}\otimes
P \otimes t^*P')\simeq
H^0(A,t_*(\omega_Y\otimes P)\otimes P')\neq
0.$$

Hence the composition of morphisms
$V^0(\omega_Y)\hookrightarrow\PA_Y\rightarrow \PK$ is surjective.
There exists an irreducible component  $Q+\PB$ of $V^0(\omega_{Y})$,
where $Q$ is a torsion point of $\PA_Y$ and $\PB$ is an abelian
subvariety of $\PA_Y$, such that the composition of the natural
morphisms $$Q+\PB\hookrightarrow \PA_{Y}\rightarrow \PK$$ is
surjective and then $(Q+\PB)\cap \widehat{\mu}(\PA)$ is not empty.

If $\dim V^0(\omega_Y, t)=0$, then
\begin{eqnarray}\label{dim 0}\dim (\widehat{\mu}(\PA)\cap V^0(\omega_Y))=0.\end{eqnarray}
Hence $$\dim ((Q+\PB)\cap \widehat{\mu}(\PA))=0,$$ and the projection $Q+\PB\rightarrow \PK$ is finite. Therefore,
$\codim_{\PA_Y}(Q+\PB)= \dim\PA\geq \dim Y-1$.

On the other hand, we know by the proof of Theorem 3 in \cite{el}
that $Q+\PB$ is also an irreducible component of $V^{\dim
\PA}(\omega_{Y})$. Hence $\dim \PA=\dim Y-1$ and the natural
morphism $A_Y\xrightarrow{(p, \mu)} B\times A$ is an isogeny. We
then consider
\begin{eqnarray*}
\xymatrix{Y\ar[drr]_{p_B}\ar[rr]^{a_{Y}}&& A_{Y}\ar[d]^{p}\ar[r]^{\mu}&A \\
&& B.}
\end{eqnarray*}
By Proposition \ref{fiberdimension}, the image $p_B(Y)$ is of dimension $1$. We take the
Stein factorization of $p_B$: $$Y\xrightarrow{\gamma} Z\rightarrow
B.$$ Since a general fiber $F$ of $\gamma$ is of general type and
$(\mu\circ a_Y)|_{F}: F\rightarrow A$ is generically finite over its image, by \cite[Theorem 1]{CH}, there exists a positive dimensional component
$T\subset V^0(\omega_F, (\mu\circ a_Y)|_{F})$. Moreover, if $P\in T$ is a torsion line bundle,
$\gamma_*(\omega_{Y/Z}\otimes a_{Y}^*\mu^*P)$ is a non-zero nef
vector bundle on $Z$ (see \cite[Corollary 3.6]{V} \footnote{This is stated there for $P$ to be trivial, and the general case follows by the \'etale covering trick.}).

If $Z$ is a smooth curve of genus $\geq 2$, then by Riemann-Roch, we
deduce that $$H^0(Y, \omega_{Y}\otimes a_{Y}^*\mu^*P)\neq 0,$$ for any torsion line bundle $P\in T$. Hence by semi-continuity,
$$\widehat{\mu}(T)\subset \widehat{\mu}(\widehat{A})\cap
V^0(\omega_{Y})\hookrightarrow \PA_{Y},$$ which contradicts
(\ref{dim 0}).

If $Z$ is an elliptic curve, then since $Z$ generates $B$, $Z$ is
isogenous to $B$. Then $Q+\PB$ is an irreducible component of
$V^0(\omega_Y)$ of dimension $1$. This is impossible by
Proposition 3.6 in \cite{CDJ}.

We then conclude the proof of Lemma \ref{lemma 1}.
\end{proof}

\begin{prop}\label{main}Let $X$ be a smooth projective variety of general type.
Assume that $X$ is of Albanese fiber dimension $1$. Then $\dim V^0(\omega_X)\geq 1$.
\end{prop}

\begin{proof}
We take  the Stein factorization of $a_X$:
\begin{eqnarray*}
\xymatrix{
X\ar[d]_g \ar[dr]^{a_X} \\
W\ar[r]^f& A_X.}
\end{eqnarray*}
We denote $\dim X=n\geq 2$ and denote by $m\geq 2$ the genus of a general fiber of $g$.

Since $X$ is of Albanese fiber dimension $1$ and is of general type, $a_{X*}\omega_X$ is a non-zero coherent sheaf on $A_X$ and hence is a non-zero
$GV$-sheaf by \cite[Corollary 4.2]{Hac}. Hence by \cite[Corollary 3.2]{Hac}, we have $$V^0(a_{X*}\omega_X)\supset V^1(a_{X*}\omega_X)\supset\cdots
V^n(a_{X*}\omega_X).$$
Therefore $V^0(\omega_X)=V^0(a_{X*}\omega_X)\neq \emptyset$ by \cite[Lemma 2.1]{CH}.

We then argue by contradiction. Assume that
$V^0(\omega_X)=V^0(a_{X*}\omega_X)$ is a union of finite number of points.
Since $a_{X*}\omega_X$ is a $GV$-sheaf and $V^0(a_{X*}\omega_X)$ is a union of finite number of points,
by \cite[Example 3.2]{muk2}, $a_{X*}\omega_X$ is a homogeneous
vector bundle on $A_X$. Hence $V^{\dim A_X}(a_{X*}\omega_X)\neq
\varnothing$, therefore $\dim A_X=\dim a_X(X)=n-1$.

There are three steps to deduce a contradiction.

First step, we claim that $a_X$ is a fibration. Hence $f$ is an isomorphism and $R^1a_{X*}\omega_X=\cO_{A_X}$.

We take $n-2$ general very ample divisors $H_i$, $1\leq i\leq n-2$
on $A_X$. Let the smooth curve $C$ be the intersection of $H_i$ on
$A_X$. By performing a base change via the morphism $C\hookrightarrow A_X$ to the above
diagram, we have
\begin{eqnarray*}\label{diagram1}
\xymatrix{
X_C\ar[d]_{g_C} \ar[dr]^{h_C} \\
W_C\ar[r]^{f_C}& C.}
\end{eqnarray*}
By a Bertini type theorem (see \cite[Theorem 1.7.1]{BS}), both $X_C$ and $W_C$ are smooth projective varieties.
Then $h_{C*}\omega_{X_C/C}=a_{X*}\omega_X|_C$.

If $\deg f>1$, since $a_X$ is the universal morphism to an abelian
variety, $f$ is ramified in codimension $1$. Hence $f_C:
W_C\rightarrow C$ is a ramified cover. We know that
$g_{C*}\omega_{X_C/W_C}$ is a nef vector bundle on $W_C$ (see
\cite{V}). Then by Riemann-Roch, we conclude that
\begin{eqnarray*}\deg h_{C*}\omega_{X_C/C}=\deg f_{C*}\big((g_{C*}\omega_{X_C/W_C})\otimes \omega_{W_C/C}\big)>0
\end{eqnarray*}
which contradicts the fact that $\deg (a_{X*}\omega_X|_C)=0$. We finish the proof of the first step.

We then write $a_{X*}\omega_X=\oplus_{i=1}^s\cV_{P_i}$, where
$\cV_{P_i}$ is the tensor product of $P_i\in\Pic^0(A_X)$ with a
unipotent vector bundle on $A_X$ and $\sum_{i=1}^s\rank(\cV_i)=m$.
Notice that each $P_i$ is a torsion line bundle (see \cite{sim}).

Second step, we claim that $P_i=\cO_X$ for each $i$. In other words, each $\cV_i$ is a unipotent vector bundle.

We still argue by contradiction to prove this claim.

If there exists a non-trivial $P_i$, then $H^{n-1}(A_X, \cV_{P_i}\otimes P_i^{*})\neq 0$ and hence $H^{n-1}(X, \omega_X\otimes P_i^*)\neq 0$.
 Take the \'etale cover $\pi:\wA_X\rightarrow A_X$ induced by the group generated by all $P_j$. Let $\pi_X: \wX\rightarrow X$ be the induced \'etale cover.
We notice that
\begin{eqnarray*}q(\wX)=h^{n-1}(\wX, \omega_{\wX})\geq h^{n-1}(X, \omega_X)+h^{n-1}(X, \omega_X\otimes P_i^*)>q(X).
\end{eqnarray*}
We have the following commutative diagram:
\begin{eqnarray*}
\xymatrix{
&&\wX\ar@/_2pc/[ddll]_{a_{\wX}}\ar[dl]^{\tau}\ar[r]^{\pi_X}\ar[dd]^{\wa_X} & X\ar[dd]^{a_X}\\
&a_{\wX}(\wX)\ar[dr]^{\upsilon}\ar@{^{(}->}[dl]\\
A_{\wX}\ar[rr]^{\mu}&&\wA_X\ar[r]^{\pi} & A_X.}
\end{eqnarray*}
Since $\wa_X$ is a fibration, either $\tau$ is generically finite
and $\wX$ is then of maximal Albanese dimension, or $\tau$ is a
fibration and $\upsilon$ is an isomorphism. The latter is
impossible, since $a_{\wX}(\wX)$ generates the whole abelian variety
$A_{\wX}$. In the first case, $\wX$ is of maximal Albanese dimension
and $\wa_X$ is surjective of relative dimension $1$. We then apply
Lemma \ref{lemma 1} and deduce that $\dim V^0(\omega_{\wX},
\wa_X)\geq 1$. This is impossible because $\dim V^0(\omega_{\wX},
\wa_X)=\dim V^0(\omega_X)=0$.

We then conclude the proof of the second step.

In the last step, we are going to deduce a contradiction.
Since $\cV_i$ is a unipotent vector bundle, $h^{n-1}(A_X, \cV_i)\geq 1$. We then have (\cite[Theorem 3.1]{kollar2})
\begin{eqnarray*}n-1=h^{n-1}(X, \omega_X)&=&h^{n-1}(A_X, a_{X*}\omega_X)+h^{n-2}(A_X, R^1a_{X*}\omega_X)\\
&=& \sum_{i=1}^s h^{n-1}(A_X, \cV_i)+h^{n-2}(A_X, \cO_{A_X})\\
&\geq & s+n-1,
\end{eqnarray*}
which is impossible.

We then conclude the proof of Proposition \ref{main}.
\end{proof}

We are now able to prove a slightly more general version of Theorem \ref{theo}:
\begin{theo}\label{maintheorem1}
 Let $X$ be a smooth projective variety of general type and let $\alpha: X\rightarrow A$ be a morphism to an abelian variety such that $\dim \alpha(X)=\dim X-1$.
Then the translates through $0$ of all irreducible components of $V^0(\omega_X, \alpha)$ generate $\PA$.
\end{theo}

\begin{proof}
By Theorem \ref{structure}, we may write $$V^0(\omega_X, \alpha)=\cup_i(Q_i+\PB_i),$$
where $Q_i\in\PA$ is a torsion point and $\PB_i$ is an abelian subvariety of $\PA$ for each $i$.
Let $\PT\hookrightarrow \PA$ be the abelian subvariety generated by all $\PB_i$'s.

Assume that $\PT$ is a proper subvariety of $\PA$. We consider the dual and get
\begin{eqnarray*}
\xymatrix{F\ar@{^{(}->}[d]\ar[rr]^h&& K \ar@{^{(}->}[d]^j\\
X\ar[dr]_f\ar[r]& \alpha(X)\ar[d]\ar@{^{(}->}[r]& A\ar[d]^p\\
& f(X)\ar@{^{(}->}[r]& T.}
\end{eqnarray*}
Let $F$ be an irreducible component of a general fiber of $f$,
$K$ be the translate of the kernel of $p$ containing $\alpha(F)$, and $h: F\rightarrow K$ be the
restriction of $\alpha$ to $F$. Notice that
$\dim F-\dim
h(F)=1$. We now study the locus $$V^0(\omega_F, h):=\{P\in\PK\mid
H^0(F, \omega_F\otimes h^*P)\neq 0\}.$$

We know that for any $P\in \PA$, $f_*(\omega_X\otimes P)$ is a
$GV$-sheaf (possible $0$) on $T$. Hence $V^0(\omega_X, \alpha)$ is mapped onto
 $V^0(\omega_F, h)$ via the morphism $\PA\xrightarrow{\widehat{j}} \PK$.
Since $\PT$ is generated by all $\PB_i$'s, for each irreducible component $Q_i+\PB_i$ of $V^0(\omega_X, \alpha)$,
$\widehat{j}(Q_i+\PB_i)$ is a torsion point of $\PK$.
Hence $V^0(\omega_F, h)$ is a
union of finite number of points.

On the other hand, we consider
\begin{eqnarray*}\xymatrix{F\ar@/^2pc/[rr]^h\ar[r]_{a_F} &A_F\ar@{->}[r]_{\pi}&  K.}\end{eqnarray*}
Since $\dim V^0(\omega_F, h)=0$, then $\Pp(\PK)\cap V^0(\omega_F)$ is a union of finite number of points in
$\PA_F$ and $\Pp: \PK\rightarrow \PA_F$ is an isogeny onto its image. By Proposition \ref{main}, $\Pp$ is not surjective.
Since an irreducible componenet of a general fiber of $h$ is a smooth projective
curve of genus $>1$, as we have seen in the proof of Lemma \ref{lemma 1},
the composition of morphisms
 $V^0(\omega_F)\hookrightarrow \PA_F\rightarrow \PA_F/\Pp(\PK)$ is surjective. Hence there exists a positive dimensional irreducible component $P_0+\PB$
 of $V^0(\omega_F)$ of codimension
equal to $\dim \Pp(\PK)\geq \dim F-1$.

If $a_F$ is generically finite, we can simply apply Lemma \ref{lemma 1}
to get a contradiction. If $F$ is of Albanese fiber dimension $1$, we apply Proposition \ref{fiberdimension} to conclude that the image of
the natural morphism $F\xrightarrow{a_F} A_F\twoheadrightarrow  B$ is a point, which is impossible since $a_F$ is the Albanese morphism of $F$.

Hence $\PT=\PA$ and we are done.
\end{proof}

For surfaces $X$ of Albanese fiber dimension $1$ and of general type, since $\chi(X,
\omega_X)>0$ (see \cite[Theorem X.4]{B}), we always have $V^0(\omega_X, a_X)=\Pic^0(X)$. In
higher dimensions, we can construct varieties of Albanese fiber
dimension $1$ and of general type with $V^0(\omega_X, a_X)$ a proper
subset of $\PA_X$ in the same way as in the construction of Ein-Lazarsfeld
threefold \cite[Example 1.13]{el}.

\begin{exam}\label{negative-char}
Let $C_i\rightarrow E_i$ be ramified double covers over elliptic curves $E_i$ for $i=1, 2$,
with associated involutions $\iota_i$. Let $C_3\rightarrow
\mathbb{P}^1$ be a ramified double cover, with associated involution
$\iota_3$. Let $X$ be a desingularization of the quotient
$(C_1\times C_2\times C_3)/(\iota_1, \iota_2, \iota_3)$. Hence $X$
is of general type and the natural morphism $X\rightarrow E_1\times
E_2$ is the Albanese morphism of $X$, hence $X$ is of Albanese fiber
dimension $1$. Moreover,
$$V^0(\omega_X, a_X)=(\{\cO_{E_1}\}\times\PE_2)\cup (\PE_1\times\{\cO_{E_2}\}).$$

We can check that $\chi(X, \omega_X)=-(g(C_1)-1)(g(C_2)-1)<0$, while the holomorphic Euler characteristic is always $\geq 0$ for varieties of maximal Albanese dimension (\cite{gl1}). Furthermore, we see that the intersection of the two irreducible components of $V^0(\omega_X)$ is $0$, while we know that if $Y$ is a variety of general type and of maximal Albanese dimension, then the intersection of any two maximal components of $V^0(\omega_Y)$ is of dimension $\geq 1$ (\cite[Proposition 1.7]{jia}).
 \end{exam}
\section{Non-general type case}
\subsection{}
In this subsection, we assume that $X$ is a smooth projective variety of Kodaira dimension $\kappa(X)\geq 0$ and  $\alpha: X\rightarrow A$ is a
morphism from $X$ to an abelian variety.
We denote by $I_X: X\rightarrow I(X)$ a model of the Iitaka fibration of $X$ with $I(X)$ a smooth projective variety. Let $F$ be a general fiber of $I_X$.
Then $F$ is a smooth projective variety of Kodaira dimension $0$. By Kawamata's theorem (\cite[Theorem 1]{kaw}),
 $\alpha(F)$ is a translate of an abelian subvariety $K$ of $A_X$ (independent of the choice of $F$).

 We consider the following commutative diagram:
\begin{eqnarray*}
\xymatrix{
X\ar[r]^{\alpha}\ar[d]_{I_X}& A\ar[d]^{\pr}\\
I(X)\ar[r]^{\alpha_I}& A/K.}
\end{eqnarray*}
The following lemma is probably known to experts, but for lack of references, we give a proof here.
\begin{lemm}\label{V0K0}
 For any irreducible component $Q+\PB$ of $V^0(\omega_X, \alpha)$,
where $Q$ is a torsion point of $\PA$ and $\PB$ is an abelian subvariety of $\PA$, we have $\PB\subset
\widehat{\pr}(\widehat{A/K})$.  In particular, if the translates through $0$ of all irreducible
components of $V^0(\omega_X, \alpha)$ generate $\PA$, then $\alpha$ factors  through the Iitaka fibration $I_X$ of $X$.
\end{lemm}

\begin{proof}
 We argue by contradiction. We consider the exact sequence
 $$0\rightarrow \widehat{A/K}\xrightarrow{\widehat{\pr}}  \PA\xrightarrow{\lambda} \PK\rightarrow 0.$$

 For a general fiber $F$ of $I_X$, we notice that  $\lambda(V^0(\omega_X, \alpha))\subset V^0(\omega_F, \alpha|_F)$.
 If there exists an irreducible component  $Q+\PB$ of $V^0(\omega_X, \alpha)$ such that $\PB\nsubseteq
\widehat{\pr}(\widehat{A/K})$, then $\lambda(\PB)$ is a positive dimensional abelian subvareity of $\PK$.
Hence $V^0(\omega_F, \alpha|_F)\supset \lambda(Q)+\lambda(\PB)$ has
positive dimension. Moreover, we know by Kawamata's theorem (\cite[Theorem 1]{kaw}) that
$\alpha|_F: F\rightarrow \alpha(F)$ is an algebraic fiber space. Hence the natural morphism $$(\alpha|_F)^*: \Pic^0(\alpha(F))=\PK\rightarrow \Pic^0(F)$$ is
an embedding.
 Let $M\geq 1$ be the order of $ \lambda(Q)$ in $\PK$. Then by repeating Clifford's lemma (see for instance \cite[Lemma 2.14]{hp}), we have $$P_{2M}(F)=\dim H^0(F, \omega_F^{2M})\geq (2M-1)\dim \lambda(\PB)+
1>1,$$ which contradicts the fact that $\kappa(F)=0$.
\end{proof}

Now if $\dim \alpha(X)=\dim X-1$,  we can prove the following much more stronger theorem by induction of dimensions. The proof is very similar to
the proof of Theorem \ref{maintheorem1} and we leave it to interested readers.
\begin{theo}\label{non-generaltype} Under the above setting, the translates through $0$ of all irreducible
components of $V^0(\omega_X, \alpha)$ generate $\widehat{\pr}\widehat{A/K}$.
\end{theo}

\subsection{}
Surfaces of Albanese fiber dimension $1$ and of Kodaira dimension $0$ are called hyperelliptic surfaces. They are completely classified: there are seven families of hyperelliptic surfaces
(see \cite[Chapter VI]{B}). In higher dimensions, it was proved by Kawamata (\cite[Theorem 15]{kaw}) that a variety of Albanese fiber dimension $1$ and of Kodaira dimension $0$ is birationally an
\'etale fiber bundle over its Albanese variety.
It is interesting to observe that the elementary method which we use to prove Proposition \ref{main} gives an easy proof of this fact and it does not rely on the
difficult addition theorem of Viehweg (\cite[Theorem I]{V1}).

 \begin{lemm}Let $F$ be a smooth projective variety of Albanese fibre dimension $1$ and of Kodaira dimension $0$.
Then there exists an abelian variety $B$, and an elliptic curve $E$ such that  $F$ is birational to $(B\times E)/H$,
where $H$ is a group of translations of $B$ acting on $E$ such that $E/H=\mathbb{P}^1$. See below the classification of $H$.
\end{lemm}
\begin{proof}
 Since  $F$ is a smooth projective variety of Albanese fibre dimension $1$ and of Kodaira dimension $0$, the Albanese morphism $a_F$ of $F$ is an algebraic fiber space
(see \cite[Theorem 1]{kaw}) and a general
fiber of $a_F$ is a curve of genus $1$. We then consider the rank $1$ sheaf
$a_{F*}\omega_F$. Since $\dim V^0(a_{F*}\omega_F)=\dim V^0(\omega_F, a_F)=0$ and $a_{F*}\omega_F$ is a $GV$-sheaf on $A_F$, we conclude that $a_{F*}\omega_F=P$ is a torsion line
bundle on $A_F$.

Let $n=\dim F$, then $q(F)=\dim A_F=n-1$. Moreover $q(F)=h^{n-1}(A_F, a_{F*}\omega_F)+h^{n-2}(A_F, R^1a_{F*}\omega_F)=h^{n-1}(A_F, a_{F*}\omega_F)+n-1$.
Hence $P\neq  \cO_F$. Let $G$ be the subgroup of $\PA_F$ generated by $P$. Then by considering the \'etale cover
$\pi: \widetilde{F}\rightarrow F$ induced by $G$, we find that
 $q(\widetilde{F})=n$. We now consider the  Albanese morphism $a_{\widetilde{F}}: \widetilde{F}\rightarrow A_{\widetilde{F}}$,
 which is naturally $G$-equivariant.
 Since the Kodaira dimension of $\widetilde{F}$ is also $0$, by \cite[Theorem 1]{kaw}, we know that $\widetilde{F}$ is
 birational to its Albanese variety $A_{\widetilde{F}}$. Therefore  $G$ also acts faithfully on $A_{\widetilde{F}}$.
Considering the following commutative diagram
 \begin{eqnarray*}
 \xymatrix{
 \widetilde{F}\ar[rr]^{\pi}\ar[d]_{a_{\widetilde{F}}}&& F\ar[d]_{\tau}\\
 A_{\widetilde{F}}\ar[rr]^{\widetilde{\pi}} && F':=A_{\widetilde{F}}/G,
 }
 \end{eqnarray*}

since both $\pi$ and $\widetilde{\pi}$ are of degree $|G|$ and $a_{\widetilde{F}}$ is birational, $\tau$ is also a birational morphism.
 Notice that $F'=A_{\widetilde{F}}/G$ has at worst quotient singularities. Hence $\tau$ induces an isomorphism $\pi_1(F)\rightarrow \pi_1(F')$
 (see \cite[Theorem 7.8]{K3}) and we have $$\pi_1(F')/\tau_*\pi_*\big(\pi_1( \widetilde{F})\big)\simeq \pi_1(F)/\pi_*\big(\pi_1( \widetilde{F})\big)\simeq G.$$
 Then the morphism $\tau\circ \pi: \widetilde{F}\rightarrow F'$ factors through the \'etale cover
 $$\widetilde{\pi'}:\widetilde{F'}\rightarrow F'$$induced by the quotient $\pi_1(F')/\tau_*\pi_*\big(\pi_1( \widetilde{F})\big)\simeq G$.

 Hence by the universal property of the Stein factorization of $\tau\circ \pi$, we conclude that $A_{\widetilde{F}} \simeq \widetilde{F'}$ and $\widetilde{\pi}$ is an
 \'etale morphism. Therefore $F'$ is a smooth projective variety and is a good minimal model of $F$.

Now we may and will assume that $\widetilde{F}\simeq A_{\widetilde{F}}$ and $F\simeq F'$. Since the composition of morphisms
$$\widetilde{F}\xrightarrow{\pi} F\xrightarrow{a_F} A_F$$ is a quotient of abelian varieties and $\pi$ is an \'etale morphism, $a_F$ is a smooth morphism. Hence
by \cite[Proposition VI.8]{B} \footnote{We notice that \cite[Proposition VI.8]{B} applies directly when $A_F$ is an elliptic curve and the same proof works when
$A_F$ is an abelian variety.}, there exists an \'etale cover $B\rightarrow A_F$ and an elliptic curve $E$ such that
$F\simeq (B\times E)/H$, where $H$ is the Galois group of $B\rightarrow A_F$. Moreover, by the same proof of \cite[Lemma VI.10]{B}, we may assume that $H$ acts on $B\times E$
diagonally. If $q(F)=\dim F-1\geq 2$, the action of $H$ on $E$ is one of the following
(by the same method of \cite[List VI.20]{B}):
 \begin{itemize}
\item[(1)] $H=\mathbb{Z}/2\mathbb{Z}$ acting on $E$ by symmetry;
\item[(2)] $H=\mathbb{Z}/2\mathbb{Z}\times \mathbb{Z}/2\mathbb{Z}$ acting on $E$ by $x\rightarrow -x$, $x\rightarrow x+\epsilon$ ($\epsilon\in E_2$, where $E_2$
is the group of points of $E$ of order $2$);
\item[(3)] $H=\mathbb{Z}/2\mathbb{Z}\times \mathbb{Z}/2\mathbb{Z}\times \mathbb{Z}/2\mathbb{Z}$ acting on $E$ by $x\rightarrow -x$,
$x\rightarrow x+\epsilon_1$, $x\rightarrow x+\epsilon_2$ ($\epsilon_1, \epsilon_2\in E_2$);
\item[(4)] $H=\mathbb{Z}/4\mathbb{Z}$ acting on $E=E_i:=\mathbb{C}/(\mathbb{Z}\oplus\mathbb{Z}i)$ by $x\rightarrow ix$;
\item[(5)] $H=\mathbb{Z}/4\mathbb{Z}\times \mathbb{Z}/2\mathbb{Z}$ acting on $E=E_i$ by $x\rightarrow ix$,
$x\rightarrow x+\frac{1+i}{2}$;
\item[(6)] $H=\mathbb{Z}/3\mathbb{Z}$ acting on $E=E_{\rho}:=\mathbb{C}/(\mathbb{Z}\oplus\mathbb{Z}\rho)$ (where $\rho^3=1$) by $x\rightarrow \rho x$;
\item[(7)] $H=\mathbb{Z}/3\mathbb{Z}\times\mathbb{Z}/3\mathbb{Z}$ acting on $E=E_{\rho}$ by $x\rightarrow \rho x$, $x\rightarrow x+\frac{1-\rho}{3}$;
\item[(8)] $H=\mathbb{Z}/3\mathbb{Z}\times\mathbb{Z}/3\mathbb{Z}\times \mathbb{Z}/3\mathbb{Z}$ acting on $E=E_{\rho}$ by $x\rightarrow \rho x$,
$x\rightarrow x+\frac{1-\rho}{3}$, $x\rightarrow x-\frac{1-\rho}{3}$;
\item[(9)] $H=\mathbb{Z}/6\mathbb{Z}$ acting on $E=E_{\rho}$ by $x\rightarrow -\rho x$.
 \end{itemize}
We notice that the cases $(3)$ and $(8)$ do not occur when $\dim F=2$.

\end{proof}\

\section{Pluircanonical maps}
In this section we deduce from Theorem \ref{theo} a quick proof of the birationality of the $4$-canonical map of varieties of general type and of Albanese
 fibre dimension $1$, which generalizes a result of Chen and Hacon \cite[Corollary 4.3]{ch2}. We use the strategy described in \cite[Section 6]{pp2},
which is used by the authors of \cite{J} and  \cite{T} to study pluricanonical maps of varieties of maximal Albanese dimension.

Throughout this section we will always denote by the following commutative diagram \begin{eqnarray}\label{setting}
 \xymatrix{
X\ar[d]_f\ar[dr]^{a_X}\\
Y\ar[r]^t & A_X}
\end{eqnarray}
a modification of the Stein factorization of the Albanese morphism of $X$ such that $Y$ is a smooth projective variety.
We denote by $F$ a general fiber of $f$ and, for a point $z\in Y$, we denote by $F_z$  the fiber of $f$ over $z$.
We will say the Iitaka model of $(X, K_X)$ dominates $Y$ if there exists an integer $M>0$ and an ample divisor $A$ on $Y$ such that
$H^0(X, MK_X-f^*A)\neq 0$ (see \cite[Proposition 1.14]{mor}).

\begin{lemm}\label{rest}
Assume that $X$ is a smooth projective variety and the Iitaka model of $(X, K_X)$ dominates $Y$.  Assume that $K_F$ is semiample. Then the co-support of the multiplier ideal $\cJ(||kK_X||)$
is disjoint from a general fiber $F$ for any $k\geq 1$.
\end{lemm}

\begin{proof}
For simplicity we let $k=1$. The argument works for all $k\geq 1$.

Since the Iitaka model of $(X, K_X)$ dominates $Y$, we can take an ample divisor $A$ on Y and an integer $M>0$  such that there exists an effective divisor
$D\in |MK_X-f^*A|$. By the subadditivity theorem (see \cite[Corollary 11.2.4]{laz}), we have $$\cJ(||K_X||)^m\supset \cJ(||mK_X||), $$ for any integer $m>0.$
Therefore we have for any $m>M$
\begin{eqnarray*}
 \cJ(||K_X||)^m\supset \cJ(||mK_X||)&=&\cJ(||(m-M)K_X+f^*A+D||)\\
  &\supset& \cJ(||(m-M)K_X+f^*A||)\otimes\cO_X(-D).
\end{eqnarray*}
On the other hand, Viehweg's weak positivity theorem (see \cite[Theorem III]{V1}) implies that for $N\gg0$, the restriction $$H^0(X, \cO_X(N(m-M)K_{X/Y}+Nf^*A)
\rightarrow H^0(F,\cO_F(N(m-M)K_{F})$$ is surjective. Since $Y$ is of maximal Albanese dimension, $K_Y$ is an effective divisor. Thus for all $N\gg0$
the restriction
$$H^0(X, \cO_X(N(m-M)K_{X}+Nf^*A)\rightarrow H^0(F,\cO_F(N(m-M)K_{F})$$ is surjective.

Therefore we conclude by \cite[Theorem 9.5.35]{laz} that for a general fiber $F$,
\begin{eqnarray*}
\cJ(||(m-M)K_X+f^*A||)\mid_F=\cJ(||(m-M)K_F||)\simeq\cO_F,
 \end{eqnarray*}
where the last isomorphism holds because $K_F$ is semiample.

Combining all the isomorphisms, we get $\cJ(||K_X||)^m\mid_F\supset \cO_F(-D)$ for all $m>M$. Hence $\cJ(||K_X||)\mid_F=\cO_F$  and the co-support of the multiplier ideal $\cJ(||K_X||)$
is disjoint from a general fiber $F$.
\end{proof}

We employ a trick of Tirabassi (\cite{T}) to show that
\begin{prop}\label{M-regular} Under the setting of diagram (\ref{setting}), we assume moreover that
\begin{itemize}
\item[1)] $2K_F$ is globally generated;
\item[2)] the translates through $0$ of all irreducible components of $V^0(\omega_X)$ generate $\Pic^0(X)$;
\end{itemize}
then there exists an open dense subset $U$ of $X$ such that for any $x\in U$, the sheaf $$a_{X*}(\cI_x\otimes \cO_X(2K_X)\otimes \cJ(||\omega_X||))$$ is $M$-regular.
\end{prop}
\begin{proof}
We notice that 1) implies that $K_F$ is semiample and, by Lemma \ref{V0K0}, 2) implies that the Iitaka model of $(X, K_X)$ dominates $Y$.
Hence we conclude by  Lemma \ref{rest}  that the co-support of
$\cJ(||\omega_X||)$ does not dominate $Y$. Hence there exists an
open dense subset $Y_0$ of $Y$ such that $V:=f^{-1}Y_0$ is disjoint from the co-support of
$\cJ(||\omega_X||)$, $f: V\rightarrow Y_0$ is smooth, $t$ is finite on $Y_0$, and $2K_F$ is globally generated for any fiber $F$ of $f$ over $Y_0$.
Then for any $x\in V$, we have an exact sequence
 \begin{multline*}0\rightarrow a_{X*}(\cI_x\otimes \cO_X(2K_X)\otimes\cJ(||\omega_X||))\rightarrow a_{X*}(\cO_X(2K_X)\otimes\cJ(||\omega_X||))
 \\\rightarrow \mathbb{C}_{a_X(x)}\rightarrow 0.\end{multline*}

By a variant of Nadel vanishing theorem
(see \cite[10.15]{K1} or \cite[Lemma 2.1]{j1}), we have
  $H^i(A_X,  a_{X*}(\cO_X(2K_X)\otimes\cJ(||\omega_X||))\otimes P)=0$, for any $i\geq 1$ and $P\in\Pic^0(A_X)$.
Hence, to conclude the proof of the proposition,  all we need to prove is that $$\codim_{\PA_X}\big(V^1(a_{X*}(\cI_x\otimes \cO_X(2K_X)\otimes\cJ(||\omega_X||)))\big)\geq 2.$$

We observe from the above exact sequence that for $x\in V$, $$V^1(a_{X*}(\cI_x\otimes \cO_X(2K_X)\otimes\cJ(||\omega_X||)))=
\{P\in\PA_X\mid x\in \Bs(|2K_X+P
|)\}.$$

Now we write $$V^0(\omega_X, a_X)=\cup_{i}(Q_i+\PB_i).$$ Let $$V'=X\moins \bigcup_i\cap_{P\in Q_i+\PB_i}\Bs(|K_X+P|).$$
Then for each $x\in V'$, we can choose an open dense subset $\mathscr{U}_i$ of $Q_i+\PB_i$ for each $i$ such that for any $P\in \cup_i\mathscr{U}_i$, $x$ is
not a base point of $|K_X+P|$. Then, $x$ is not a base point of $|2K_X+Q|$, for any $Q\in \bigcup_i(\mathscr{U}_i+\mathscr{U}_i)=\bigcup_i (2Q_i+\PB_i)$.

Let $U=V\cap V'$. We then conclude that for any $x\in U$,
\begin{eqnarray*}
V^1(a_{X*}(\cI_x\otimes \cO_X(2K_X)\otimes\cJ(||\omega_X||)))\cap \bigcup_i(2Q_i+\PB_i)=\emptyset.
\end{eqnarray*}

According to hypothesis $2)$, all the $\PB_i$'s generate the whole abelian variety $\PA_X$, hence we have (see for instance \cite[Lemma 2]{T})
$$\codim_{\PA_X}\big(V^1(a_{X*}(\cI_x\otimes \cO_X(2K_X)\otimes\cJ(||\omega_X||)))\big)\geq 2.$$
\end{proof}

Now we are ready to prove the
following theorem, which is slightly more general than Theorem
\ref{theorem1.1}.

\begin{theo}\label{pluri}Let $X$ be a smooth projective variety of general type and of Albanese fiber dimension $1$.
Then, for any $P\in\Pic^0(X)$, the map $\varphi_{4K_X+P}$ induced by the linear system $|4K_X+P|$ is
a birational map of $X$.
\end{theo}
\begin{proof}
We use the setting of diagram (\ref{setting}).

By Proposition \ref{M-regular}, we can take an open dense subset
$U$ of $Y$  and an open dense subset $V\subset f^{-1}U$ such that $f^{-1}U$ is disjoint from the co-support of
$\cJ(||\omega_X||)$, $f: f^{-1}U\rightarrow U$ is smooth, $t$ is finite on $U$, and for any $x\in
V$, $a_{X*}(\cI_x\otimes
\omega_X^2\otimes \cJ(||\omega_X||))$ is $M$-regular and thus is continuously globally generated on $A_X$ by Theorem
\ref{pareschi-popa}.

Any fiber $F$ of $f$ over $U$ is a smooth projective curve of genus $g:=g(F)\geq 2$.
We notice that $|2K_F|$ induces an isomorphism of $F$ if $g \geq 3$, and if $g =2$, then $F$ is a hyperelliptic curve  with hyperelliptic involution $\tau_F$
and the linear system $|2K_F|$ induces the quotient morphism of $F$ by $\tau_F$.

 We see from the choice of $U$ that, for any $t\in U$, the natural map
 $$f_*(\cO_X(2K_X)\otimes\cJ(||\omega_X||))\otimes\mathbb{C}_t\rightarrow H^0\big(F_t, \cO_X(2K_X)\otimes\cJ(||\omega_X||)\big)\simeq H^0(F_t, \omega_{F_t}^2)$$
is an isomorphism.

For two different points $x, y\in V$, we denote respectively by $F_{f(x)}$ and $F_{f(y)}$ the fibers of $f$ over $f(x)$ and $f(y)$.

By pushing forward via $f$ the following short exact sequence:
$$0\rightarrow \cI_{x, y}\otimes \omega_X^2\otimes \cJ(||\omega_X||)\rightarrow \cI_{x}\otimes \omega_X^2\otimes \cJ(||\omega_X||)
\rightarrow \mathbb{C}_y\rightarrow 0,$$
we get the long exact sequence:
\begin{multline}\label{mxy} 0\rightarrow f_*(\cI_{x, y}\otimes \omega_X^2\otimes \cJ(||\omega_X||))
\rightarrow f_*(\cI_{x}\otimes \omega_X^2\otimes \cJ(||\omega_X||))\\ \xrightarrow{m_{x,y}} f_*\mathbb{C}_{y}\simeq\mathbb{C}_{f(y)}\rightarrow \cdots
\end{multline}

\medskip
\noindent{\bf Step 1.}
{\em  For any two different points $x, y\in V$, $m_{x,y}$ is surjective, unless $g=2$, $f(x)=f(y)$, and $y=\tau_{F_{f(x)}}(y)$.

}
The argument to prove the claim is contained in (\cite[Step 2 of the proof of Theorem 4.2]{ch2}). For reader's convenience, we repeat it here.

If $f(x)\neq f(y)$, then
$$f_*(\cI_x\otimes \omega_X^2\otimes \cJ(||\omega_X||))\otimes \mathbb{C}_{f(y)}\simeq H^0(F_{f(y)}, \omega_{F_{f(y)}}^2).$$
Since $\omega_{F_{f(y)}}^2$ is base point free, $m_{x,y}$ is surjective.

If $f(x)=f(y)$, we denote by $t=f(x)$ and notice that
\begin{multline*}
f_*(\cI_{x}\otimes \omega_X^2\otimes \cJ(||\omega_X||))\otimes \mathbb{C}_{t}\\=
\ker\big(f_*(\cO_X(2K_X)\otimes\cJ(||\omega_X||))\otimes\mathbb{C}_{t}\simeq H^0(F_{t}, \omega_{F_{t}}^2)\xrightarrow{ev_x}  \mathbb{C}_x\big),
\end{multline*}
and
\begin{multline*}
f_*(\cI_{x,y}\otimes \omega_X^2\otimes \cJ(||\omega_X||))\otimes \mathbb{C}_{t}\\=
\ker\big(f_*(\cO_X(2K_X)\otimes\cJ(||\omega_X||))\otimes\mathbb{C}_{t}\simeq H^0(F_{t}, \omega_{F_{t}}^2)\xrightarrow{(ev_x,ev_y)}  \mathbb{C}_x\oplus \mathbb{C}_y\big).
\end{multline*}
If $g \geq 3$ or $g=2$ and $\tau_{F_{t}}(x)\neq y$, the linear system $|2K_{F_{t}}|$ separates $x$ and $y$ and hence $(ev_x,ev_y)$ is surjective.
Thus the inclusion
$$f_*(\cI_{x, y}\otimes \omega_X^2\otimes \cJ(||\omega_X||))\otimes \mathbb{C}_{t}\hookrightarrow f_*(\cI_{x}\otimes \omega_X^2\otimes \cJ(||\omega_X||))\otimes
\mathbb{C}_{t}$$ is strict. Hence $m_{x, y}$ is again surjective.

\medskip
\noindent{\bf Step 2.}
{\em If $g\geq 3$, $\varphi_{4K_X+P}$ is birational for any $P\in\Pic^0(X)$.}

Since $g\geq 3$, for any two different points $x, y\in V$, we have by Step 1 the following short exact sequence:
\begin{multline*}0\rightarrow f_*(\cI_{x,y}\otimes \omega_X^2\otimes \cJ(||\omega_X||))
\rightarrow f_*(\cI_{x}\otimes \omega_X^2\otimes \cJ(||\omega_X||))\\ \xrightarrow{m_{x,y}} \mathbb{C}_{f(y)}\rightarrow 0.
\end{multline*}
Notice that $t$ is finite on $U$ and $a_X=t\circ f$. Hence we have the following short exact sequence on $A_X$:
\begin{multline}\label{multline}0\rightarrow a_{X*}(\cI_{x, y}\otimes \omega_X^2\otimes \cJ(||\omega_X||))
\rightarrow a_{X*}(\cI_{x}\otimes \omega_X^2\otimes \cJ(||\omega_X||))\\ \xrightarrow{ev_{x,y}}\mathbb{C}_{a_X(y)}\rightarrow 0.
\end{multline}
Because $a_{X*}(\cI_x\otimes \omega_X^2\otimes
\cJ(||\omega_X||))$ is continuously globally generated,
there exists an open dense subset $\cU_{x,y}$ of $\Pic^0(X)$ such that  for any $Q\in \cU_{x,y}$  the map
$$H^0(A_X, a_{X*}(\cI_{x}\otimes \omega_X^2\otimes \cJ(||\omega_X||))\otimes Q)\xrightarrow{ev_{x,y}} \mathbb{C}_{a_X(y)} $$
is surjective. In other words, for
any $Q\in \cU_{x,y}$, there exists a global
section
\begin{eqnarray*}s_Q&\in & H^0(A_X, a_{X*}(\cI_{x}\otimes \omega_X^2\otimes \cJ(||\omega_X||))\otimes Q)\\&\simeq & H^0(X, \cI_{x}\otimes \omega_X^2\otimes \cJ(||\omega_X||)\otimes Q)\\
&\subset& H^0(X, \cI_{x}\otimes \omega_X^2\otimes Q),
 \end{eqnarray*}which  does not vanish at $y$, and hence the divisor $D_Q\in |2K_X+Q|$ defined by $s_Q$ passes through $x$ but does not pass through $y$.
Therefore, for any $P\in \cU_{x,y}+\cU_{x,y}=\Pic^0(X)$, the linear system $|4K_X+P|$ separating $x$ and $y$. The map
$\varphi_{4K_X+P}$ is an injective morphism on $V$ and we are done.

\medskip
\noindent{\bf Step 3.}
{\em If $g=2$, $\varphi_{4K_X+P}$ is also birational for any $P\in\Pic^0(X)$.}

Since $g=g(F)=2$, the generic fiber of $f$ is a hyperelliptic curve. Hence there exists a birational involution $\tau$ of $X$
inducing $\tau_F$ in a
general fiber $F$. After birational modifications of $X$ (see for instance \cite[4.2]{pr}), we may assume that $\tau$ is a biregular involution of the smooth projective variety $X$. Then each connected component of the fixed loci of $\tau$ is smooth.
After blowing-up the fixed loci of $\tau$,
we may  assume that each connected component of the fixed loci of $\tau$ is a smooth divisor of $X$ and hence $Z:=X/\langle\tau\rangle$ is also smooth. We choose $U$ and $V$ as before. Moreover, we may assume that $\tau(V)=V$.
Now we have the commutative diagram
\begin{eqnarray*}
\xymatrix{X\ar[r]^{\phi}\ar[dr]_f&  Z\ar[d]^g\\
& Y,}
\end{eqnarray*}
where $\phi$ is the quotient morphism induced by $\tau$.

By Step 1, for two points $x, y \in V$ with $\tau(x)\neq y$, $m_{x,y}$ in the long exact sequence (\ref{mxy})
 is surjective. Hence the arguments in Step 2 show that $\varphi_{4K_X+P}$ separates any two
 different points $x, y\in V$ with $ \tau(x)\neq y$.

Therefore $\phi$ factors  birationally through $\varphi_{4K_X+P}$  and
$\deg\varphi_{4K_X+P}\leq 2$. In order to finish the proof, we just need
to prove that $\varphi_{4K_X+P}$ does not factor birationally through $\phi$.

Because $\tau$ induces the hyperelliptic involution on a general fiber of $f$, a general fiber $L$ of $g$ is isomorphic
to $\mathbb{P}^1$. Since $\phi$ is a finite morphism of degree $2$ between smooth projective varieties, there exists a divisor $D$ on $Z$ such that
\begin{eqnarray}\label{definitionD}\phi_*\cO_X=\cO_Z\oplus \cO_Z(-D),\\\nonumber
 \omega_X=\phi^*(\omega_Z\otimes \cO_Z(D)).
\end{eqnarray}
  Moreover, $\deg(D|_L)=3$.

By (\ref{definitionD}), we  have
\begin{eqnarray*}
 \phi_*\cO_X(4K_X)=\cO_Z(4K_Z+4D)\oplus \cO_Z(4K_Z+3D),
\end{eqnarray*}
where $\cO_Z(4K_Z+4D)$ (resp. $\cO_Z(4K_Z+3D)$) is the $\tau$-invariant part (resp. $\tau$-anti-invariant part) of $\phi_*\cO_X(4K_X)$.
Then there exist ideal sheaves $\cI_1$ and $\cI_2$ on $Z$ such that
\begin{eqnarray*}
\phi_*\big(\cO_X(4K_X)\otimes\cJ(||3K_X||)\big)&=&\big(\cO_Z(4K_Z+4D)\otimes\cI_1\big)\\&\oplus&\big(\cO_Z(4K_Z+3D)\otimes\cI_2\big).
\end{eqnarray*}

Because $\deg(\omega_Z|_L)=-2$, both $g_*(\cO_Z(4K_Z+4D))$ and
$g_*(\cO_Z(4K_Z+3D))$ are nonzero torsion-free sheaves on $Y$. Since the co-support of
$\cJ(||3K_X||)$ does not dominate $Y$, the sheaves $g_*\big(\cO_Z(4K_Z+4D)\otimes\cI_1\big)$ and
$g_*\big(\cO_Z(4K_Z+3D)\otimes\cI_2\big)$ are also nonzero. Thus, since
$$t_{*}f_*\big(\cO_X(4K_X)\otimes\cJ(||3K_X||)\big)$$ is an $IT^0$
sheaf on $A_X$, we conclude that the sheaves
\begin{eqnarray*}t_{*}g_*\big(\cO_Z(4K_Z+4D)\otimes\cI_1\big) \;\;
\textrm{and} \;\;
t_{*}g_*\big(\cO_Z(4K_Z+3D)\otimes\cI_2\big)\end{eqnarray*}
are nonzero $IT^0$
sheaves on $A_X$.

Therefore, for any $P\in\PA_X$, $H^0(Z, \cO_Z(4K_Z+3D)\otimes P)\neq 0$. Because $\cO_Z(4K_Z+3D)\otimes P$ is the $\tau$-anti-invariant part of
$\phi_*\cO_X(4K_X)\otimes P$, a nonzero global section of $\cO_Z(4K_Z+3D)\otimes P$ separates $x$ and $\tau(x)$, for $x\in V$ general. Thus
$\varphi_{4K_X+P}$ does not factor  birationally through $\phi$ and $\varphi_{4K_X+P}$ is
birational.
\end{proof}
\begin{rema}
It is well-known that the tricanonical map of a smooth projective curve of genus $\geq 2$ is an isomorphism. Recently, the authors of \cite{JLT} have shown that
the tricanonical map is birational for a  variety
of maximal Albanese dimension and of general type. These lead the authors to believe that the tricanonical map of a variety of Albanese fiber dimension $1$ and of general type
is also birational.

By the same method in the proof of Theorem \ref{pluri}, it is quite easy to show that the tricanonical map is birational if $V^0(\omega_X)=\Pic^0(X)$. Hence it leaves to analyse
the case when $V^0(\omega_X)\neq \Pic^0(X)$. One might want to mimic the argument in \cite{JLT}.
 For each irreducible component $P+\PB$ of $V^0(\omega_X)$ of codimension $k$ in $\Pic^0(X)$, we do have a natural fibration by Proposition \ref{fiberdimension}:
\begin{eqnarray*}
\xymatrix{
 X\ar[d]^f\ar[dr]^{a_X}\\
Y\ar[r]^t\ar[d]^{f_B} & A_X\ar[d]\\
X_B\ar[r]& B,}
\end{eqnarray*}
where $X_B$ is of dimension $\dim X-k-1$.
The problem here is that we don't have a good geometric explanation for the existence of $P+\PB$ unlike the irreducible component of $V^0$ of a variety of maximal Albanese dimension
(see \cite[Theorem 3.1]{CDJ}). But the sheaf $R^kf_{B*}(f_*\omega_X\otimes P)$ does have certain positive properties over $X_B$ and may be crucial to a proof of the birationality
of the tricanonical map.
\end{rema}

\end{document}